\documentclass[12pt,a4paper]{amsart}
\usepackage[dvips]{graphicx}
\usepackage{enumerate}
\usepackage{amsmath,amssymb,amsthm}
\usepackage[usenames]{color}

\theoremstyle{plain}
\newtheorem{thm}{Theorem}[section]

\newtheorem{lem}[thm]{Lemma}

\theoremstyle{definition}
\newtheorem{defn}[thm]{Definition}
\newtheorem{rem}[thm]{Remark}

\newtheorem{exa}[thm]{Example}

\def\ZZ{\mathbb Z}
\def\QQ{\mathbb Q}
\def\RR{\mathbb R}
\def\NN{\mathbb N}
\def \trd {\root 3 \of {2}}
\def \trs {\root 3 \of {4}}
\def \abc {a + b\trd + c\trs}

\def\vec#1{\mathbf{#1}}

\title{Different bases in  investigation of $\trd$}
\author{Mitja Lakner, Peter Petek,  Marjeta \v Skapin Rugelj}
\thanks{ M. Lakner,  M. \v{S}kapin Rugelj: University of Ljubljana, Faculty of Civil and Geodetic Engineering,
Jamova 2, 1000 Ljubljana, Slovenia.\\
P. Petek: University of Ljubljana, Faculty of Education,
Kardeljeva plo\v{s}\v{c}ad 16, 1000 Ljubljana, Slovenia.\\ \textit{ E-mails:}
mlakner@fgg.uni-lj.si, Peter.Petek@guest.arnes.si,
mskapin@fgg.uni-lj.si.}
\date{\today}

\begin{document}
\parindent0cm
\parskip2mm

\begin{abstract}
The present paper is in a sense a continuation of \cite{PLS},
it relies on the notation and some
results. The problem tackled in both papers is the nature of the
continued fraction expansion of $\trd$: are the partial quotients bounded
or not. Numerical experiments suggest an even stronger result on
the lines of Kuzmin statistics. Here we apply different sets of bases for the vector space $V$,
where the adjunction ring $\ZZ\lbrack\trd\rbrack$ lives. And as a result
we get a criterion for continued fraction convergents in terms of the
coefficient vector from a lattice.
\medskip

\noindent {\it Key words:} bases, cubic root, continued fractions.\\
\noindent {\it Mathematics Subject Classification (2000):} 11A55, 11K50, 11R16.

\end{abstract}

\maketitle

\section{Introduction}

Stability of invariant circles in K.A.M. theory depends on the respective
rotation number. The most stable circle has $\phi={-1+\sqrt{5}\over 2}$,
the golden mean ratio, as its rotation number, all partial quotients equal
$b_i=1$.  As in \cite{SL} boundness of partial quotients would be a bonus
in representation on a computer. Here however, experiments strongly suggest the opposite.
Also cubic irrationals are interesting in studying
quasiperiodic motion \cite{CSG}, \cite{MH}. Here we investigate $\trd$ and its
adjunction ring. It is a common belief that the partial quotients of
$\trd$ are not bounded, supported by extensive computations, but no proof.

Even more, computations suggest that their relative frequencies in the limit obey
the Kuzmin law $P(b_n=k)=\log_2{(k+1)^2\over k(k+2)}$. In \cite{RDM} several
algebraic numbers were used in computations, among them $\root 3 \of {2}, \root 4 \of{2},
\root 5 \of {2}$ and good accordance was found with Kuzmin's statistics, for
$\trd$ even too good. So later in \cite{LT} and \cite{Brj} larger samples were taken
and the anomaly seemed to disappear. We played the same game, only we had the advantage of
more sophistical computation tools that evolved in the years in between. The
experimental results supporting of the stronger hypothesis instigated us to try towards
some theoretical results.

 However questions about "big" partial quotients
still linger and tease us. Though some explanation was given in a special case,
using elliptic modular functions \cite{S}, \cite{CM}.

In this paper we work towards the proof of

\noindent{\bf Hypothesis.} The partial quotients of $\trd$ are not bounded.\\[2mm]
denoted as Hypothesis B in \cite{PLS}. Only a very limited partial result is given,
helping to recognize possible convergents and estimating the next partial quotient.

\section{Adjunction ring and ambient vector space $V$}

For convenience of the reader we repeat some definitions and notations
from \cite{PLS}.

In the adjunction ring we have the unit $\rho=1+\trd+\trs$ and its inverse
$\sigma =-1+\trd$, $\rho\sigma=1$. We span the 3-dim space $V$ on $\rho, 1, \sigma$. Obviously
the continued fraction expansions for $\trd$ and $\sigma$ differ only in the
starting partial quotient, so we rather consider approximations to $\sigma$.
And in order to find ever better approximations, we construct series of
vector space bases for $V$.

Multiplicative norm is defined in $\ZZ\lbrack\trd\rbrack$. Let
$x=\abc$, its norm is
\begin{equation}\label{norma}
N(x)=a^3+2b^3+4c^3-6abc =x\cdot x'\cdot x''
\end{equation}
 with
$$x'=a+\omega b\trd+\omega^2 c\trs,$$
$$x''=a+\omega^2 b\trd+\omega c\trs$$
and $\omega=e^{2\pi i\over 3}$ a third root of one.

Division in general leads to the corresponding field $\QQ\lbrack\trd\rbrack$.
Carrying out the rationalization of the denominator as in
$${1\over x}={x'\cdot x''\over N(x)}={a^2-2bc+(2c^2-ab)\trd+(b^2-ac)\trs\over N(x)}$$
gives the elements of $\QQ\lbrack\trd\rbrack$ in the form $\abc$ with
$a,b,c$ rational fractions.

\section{The ambient vector space $V$}

As already mentioned, instead of $(1,\trd,\trs)$ we use the algebraic basis
$(\rho,1,\sigma)$ and elements of $\ZZ\lbrack\trd\rbrack$ are expressed as
$$w=x\cdot\rho+y\cdot 1+z\cdot\sigma$$
or changing the basis
$$\abc=c\cdot\rho+(a-2c+b)\cdot 1+(b-c)\cdot\sigma.$$

Now, let $V=\RR^3$ be the 3-dimensional space endowed with the
usual scalar product $\langle\vec{a},\vec{b}\rangle$ and cross
product $\vec{a}\times\vec{b}$. Vectors can be written as triplets
$$V=\{\vec{v}=(x,y,z); x,y,z\in \RR\}$$
and define a linear mapping
$$\eta:\ZZ\lbrack\trd\rbrack\to V$$
by $\eta(x\cdot\rho+y\cdot 1+z\cdot\sigma)=(x,y,z)$, the resulting image
consisting of all vectors with integer entries, multiplication inherited from
$\ZZ\lbrack\trd\rbrack$.

Taking into account
$$
x\cdot\rho+y\cdot 1+z\cdot\sigma = (x+y-z) + (x+z)\trd + x\trs
$$
we can define the norm function in the whole $V$:
\begin{equation*}\label{delta}
\widetilde{N}(x,y,z)=(x+y-z)^3+2(x+z)^3+4x^3-6(x+y-z)(x+z)x
\end{equation*}

Multiplication with $\sigma$ will prove very important and we observe
$$\eta(\sigma\cdot w)=S\eta(w)$$
where
$$S=\left[
\begin{matrix}
0&0&1\cr
1&0&-3\cr
0&1&-3
\end{matrix}
\right].$$
which, by the way, represents a hyperbolic toral authomorphism  \cite{BD}.

\section{Interplay of different bases in $V$}

The term basis comes in several ways in mathematics. Our
discussion needs it in two appearances:
\begin{itemize}
\item as the {\it basis} of a number system ($\rho$ in our case),
\item as the {\it basis} of a vector space (different bases of $V$
here).
\end{itemize}

The first usage figures in our paper \cite{PLS}, here we are concerned
with the second one.

We can think of the basis ${\mathcal B}_0=((1,0,0),(0,1,0),(0,0,1))$ as the
canonical one. And we shall also write
${\mathcal B}_0=(\vec{s}_{-1},\vec{s}_0,\vec{s}_1)$
as we shall denote $\eta(\sigma^j)=\vec{s}_j$, and noting
$S\vec{s}_j=\vec{s}_{j+1}$ we have

\quad $\vec{s}_0=(0,1,0)$,

\quad $\vec{s}_1=(0,0,1)$,

\quad $\vec{s}_2=(1,-3,-3)$,

\quad $\vec{s}_3=(-3,10,6)$,

\quad $\vec{s}_4=(6,-21,-8)$,

\quad $\vec{s}_5=(-8,30,3)$,

\quad $\vec{s}_6=(3,-17,21)$,

and we may also need the ones with negative indices

\quad $\vec{s}_{-1}=(1,0,0)$,

\quad $\vec{s}_{-2}=(3,3,1)$,

\quad $\vec{s}_{-3}=(12,10,3)$,

\quad $\vec{s}_{-4}=(46,39,12)$,

\quad $\vec{s}_{-5}=(177,150,46)$,

\quad $\vec{s}_{-6}=(681,577,177)$,

as well as the inverse matrix
$$
S^{-1}= \left[
\begin{matrix}3&1&0\cr
              3&0&1\cr
              1&0&0
\end{matrix}
\right],$$ which we also meet in the Jacobi-Perron algorithm
\cite{LB}.

Further we define the series of bases
${\mathcal B}_j=(\vec{s}_{j-1},\vec{s}_j,\vec{s}_{j+1})$ for all integer $j$.
These are good bases for our purposes as

\begin{lem}
Elements of $\ZZ^3$ have integer coefficients
in each basis ${\mathcal B}_j$.
\end{lem}
\begin{proof} Let $v$ be an element from $\ZZ^3$.
If we multiply expansion
$$v=\alpha\vec{s}_{j-1}+\beta\vec{s}_j+\gamma\vec{s}_{j+1}$$
by integer element matrix $S^{-j}$, we get vector with integer components
$$S^{-j}v=\alpha\vec{s}_{-1}+\beta\vec{s}_0+\gamma\vec{s}_1=(\alpha,\beta,\gamma).$$
\end{proof}

Besides the series of bases ${\mathcal B}_j$ we also define the
conjugate series ${\mathcal B}^*_j$ in the following manner.

For start $\vec{s}^*_0=(1,0,0)$ and with the adjoint matrix
$$
S^*=\left[
\begin{matrix}0&1&0\cr
              0&0&1\cr
              1&-3&-3
\end{matrix}
\right] $$

we define vectors $\vec{s}^*_{j+1}=S^*\vec{s}^*_j$ for positive
and negative indices. So we have

\quad $\vec{s}^*_{-3}=(46, 12, 3)$,

\quad $\vec{s}^*_{-2}=(12,3,1)$,

\quad $\vec{s}^*_{-1}=(3,1,0)$,

\quad $\vec{s}^*_{0}=(1,0,0)$,

\quad $\vec{s}^*_{1}=(0,0,1)$,

\quad $\vec{s}^*_{2}=(0,1,-3)$,

\quad $\vec{s}^*_{3}=(1,-3,6)$.

\begin{rem}
Comparing with the vectors $\vec{s}_j=(x_j, y_j, z_j)$ we see that
$$\vec{s}^*_j=(x_{j-1},x_j,x_{j+1})=(z_{j-2},z_{j-1},z_j).$$
\end{rem}

The bases series ${\mathcal B}^*_j=(\vec{s}^*_{j-1},\vec{s}^*_j,\vec{s}^*_{j+1})$
again providing integer coefficients for $\ZZ\lbrack\trd\rbrack$. The two
bases series shall be useful in further computations.

The action of linear transformation $S$ is best understood in
terms of its eigenvalues and eigenspaces. One eigenvalue is real,
smaller then $1$, and two are complex conjugate greater than $1$

\quad $\lambda_1=\sigma=-1+\trd$,

\quad $\lambda_2=\sigma'=-1+\omega\trd=-1-{\trd\over 2}+{i\over 2}\sqrt{3}\trd$,

\quad $\lambda_3=\sigma''=-1+\omega^2\trd=-1-{\trd\over 2}-{i\over 2}\sqrt{3}\trd$.

The eigenvectors being $\vec{h}$ and $\vec{g}\pm i\vec{k}$ where

\quad $\vec{h}={1\over 6}(\trd,2-2\trd+\trs,-\trd+\trs)\doteq(0.209987,0.177926,0.05458)$,

\quad $\vec{g}={1\over 12}(-\trd,4+2\trd-\trs,\trd-\trs)\doteq(-0.104993,0.411037,-0.02729)$,

\quad $\vec{k}={\sqrt{3}\over
12}(\trd,-2\trd-\trs,-\trd-\trs)\doteq(0.181854,-0.592829,-0.410976)$.

We can also compute the rotation angle
$$\theta=\pi-\arctan {\sqrt3\root 3 \of {2}\over 2+\root 3 \of {2}}\doteq146.2^0$$

and after some computation we can express
$$\vec{s}_j=\sigma^j\vec{h}+2\rho^{j\over 2}(\vec{g}\cos(j\theta)-\vec{k}\sin(j\theta)),$$
and
$$S\vec{g}=\sqrt{\rho}(\vec{g}\cos\theta-\vec{k}\sin\theta),$$
$$S\vec{k}=\sqrt{\rho}(\vec{g}\sin\theta+\vec{k}\cos\theta).$$

The norm (\ref{norma}) takes zero value on the union of the
eigenplane $P$ spanned by
 vectors $\vec{g}$, $\vec{k}$ and the eigenline of $\vec{h}$. Except for the origin,
  there is no rational point $(x,y,z)$ of zero norm.

The basic vectors $\vec{s}_j$ with increasing positive $j$ are approaching
the invariant plane  and for negative $j$ being almost colinear
to the eigenvector $\vec{h}$.

On the other hand we can construct the eigenbasis
${\mathcal B}_e=(\vec{h},\vec{g},\vec{k})$ and the conjugate eigenbasis.

To make the conjugate eigenbasis, we compute the vector products
\begin{eqnarray*}
\vec{h}^*&=&\vec{g}\times\vec{k},\\
\vec{g}^*&=&\vec{h}\times\vec{k},\\
\vec{k}^*&=&\vec{h}\times\vec{g}
\end{eqnarray*} and they
constitute the conjugate eigenbasis ${\mathcal B}^*_e$

\quad $\vec{h}^*={-\sqrt{3}\over 36}(1+\trd+\trs,1,-1+\trd)\doteq(-0.185104,-0.0481125,-0.0125055)$,

\quad $\vec{g}^*={-\sqrt{3}\over 36}(-2+\trd+\trs,-2,2+\trd)\doteq(-0.0407668,0.096225,-0.156843)$,

\quad $\vec{k}^*={1\over 12}(\trd-\trs,0,\trd)\doteq(-0.02729,0,0.104993)$.

Later we shall need also the mixed product
$$[\vec{h},\vec{g},\vec{k}]={-\sqrt{3}\over
36}=-M\doteq-0.0481125.$$
 The plane $P^*$ of the vectors
$\vec{g}^*, \vec{k}^*$ is the invariant plane of $S^*$.

\begin{lem}The vectors $\vec{h}^*, \vec{g}^*\pm i\vec{k}^*$ are eigenvectors
of the matrix $S^*$.
\end{lem}
\begin{proof} First we find the scalar products:
$$\langle S^*\vec{h}^*,\vec{g}\rangle=\langle\vec{g}\times\vec{k},S\vec{g}\rangle=
\langle\vec{g}\times\vec{k},\sqrt{\rho}(\vec{g}\cos\theta-\vec{k}\sin\theta)\rangle=0,$$
$$\langle S^*\vec{h}^*,\vec{k}\rangle=\langle\vec{g}\times\vec{k},S\vec{k}\rangle=
\langle\vec{g}\times\vec{k},\sqrt{\rho}(\vec{g}\sin\theta+\vec{k}\cos\theta)\rangle=0.$$

So it is clear $S^*\vec{h}^*$ is orthogonal to both $\vec{g}$ and $\vec{k}$ and
therefore colinear to $\vec{h}^*$ itself, therefore an eigenvector with a real eigenvalue,
the only one being $\sigma$, therefore $S^*\vec{h}^*=\sigma\vec{h}^*$. The other two
cases demand a little more work to tell apart the two complex eigenvalues, of course
unless we want to go into direct computation.
\end{proof}

We can also express the vectors $\vec{s}^*_j$ in terms of the conjugate
eigenbasis
$$\vec{s}^*_j=-2\sqrt{3}\trd\left(\sigma^j\vec{h}^*+\rho^{j\over 2}(\vec{g}^*\cos(j\theta-{\pi\over 3})
-\vec{k}^*\sin(j\theta-{\pi\over 3}))\right)$$ and infer a
connection between the basis and conjugate basis via the matrix
$$T=\left[
\begin{matrix}3&1&0\cr
              1&0&0\cr
              0&0&1
\end{matrix}
\right],$$

namely
$$\vec{s}^*_j=T\vec{s}_j$$

and also connecting the eigenbasis with conjugate eigenbasis

\quad $T\vec{h}=-2\sqrt{3}\trd\vec{h}^*$,

\quad $T\vec{g}={\sqrt{3}\trd\over 2}(-\vec{g}^*-\sqrt{3}\vec{k}^*)$,

\quad $T\vec{k}={\sqrt{3}\trd\over 2}(\sqrt{3}\vec{g}^*-\vec{k}^*)$.

We shall also need some scalar and cross products of the basis vectors.

\begin{lem} The scalar products of basis vectors are as follows:
$$\langle\vec{s}^*_n,\vec{s}_k\rangle=\langle\vec{s}^*_0,\vec{s}_{k+n}\rangle$$
with

\quad $\langle\vec{s}^*_0,\vec{s}_{-2}\rangle=3$,

\quad $\langle\vec{s}^*_0,\vec{s}_{-1}\rangle=1$,

\quad $\langle\vec{s}^*_0,\vec{s}_{0}\rangle=0$,

\quad $\langle\vec{s}^*_0,\vec{s}_{1}\rangle=0$,

\quad $\langle\vec{s}^*_0,\vec{s}_{2}\rangle=1$,

\quad $\langle\vec{s}^*_0,\vec{s}_{3}\rangle=-3$.
\end{lem}
\begin{proof} Since
$$\langle\vec{s}^*_n,\vec{s}_k\rangle=\langle {S^*}^n\vec{s}^*_0,\vec{s}_k\rangle=
\langle\vec{s}^*_0,S^n\vec{s}_k\rangle=\langle\vec{s}^*_0,\vec{s}_{k+n}\rangle$$
we only need to read off the first component of $\vec{s}_j$ as
$\vec{s}^*_0=(1,0,0)$.
\end{proof}

\begin{lem} For two consecutive basis vectors we have the cross product
$$\vec{s}_{-j}\times\vec{s}_{-j+1}=\vec{s}^*_j$$
and if we jump by one index
$$\vec{s}_{-j-1}\times\vec{s}_{-j+1}=-\vec{s}^*_{j-1}+3\vec{s}^*_j.$$
\end{lem}

\begin{proof} Setting unknown coefficients $\alpha,\beta,\gamma$
$$\vec{s}_{-j}\times\vec{s}_{-j+1}=\alpha\vec{s}^*_{j-1}+\beta\vec{s}^*_j+\gamma\vec{s}^*_{j+1}$$
and taking scalar products in turn with $\vec{s}_{-j-1},\vec{s}_{-j},\vec{s}_{-j+1}$
we get
$$1=\alpha\cdot 3+\beta\cdot 1+\gamma\cdot 0$$
$$0=\alpha\cdot 1+\beta\cdot 0+\gamma\cdot 0$$
$$0=\alpha\cdot 0+\beta\cdot 0+\gamma\cdot 1$$
and from here $\alpha=\gamma=0$ and $\beta=1$. Likewise we prove the second formula.
\end{proof}

\section{The shortest coefficient vector and convergents}

The vector $(0,p,-q)=\eta(p-q\sigma)$ can as any vector be expanded in any
basis ${\mathcal B}_j$ and with integer coefficients.

In \cite{PLS}, using $\rho$ basis, we expressed $(0,p,-q)$ with a wider choice of vectors, but
coefficients being limited to 0,1,2 or 3: $(0,p,-q)=\sum_{j=k}^na_j\vec{s}_j$.
Whereas here, since each time only 3 vectors form the basis, we allow all integer coefficients.

Should ${p\over q}$ be a convergent
to $\sigma$, we can control the size of the these coefficients, provided $j$ has
been chosen appropriately.

\begin{defn} Let ${p\over q}$ be a convergent to $\sigma$ and $\vec{a}=(a_1,a_2,a_3)$ the
coordinates of
$\eta(p-q\sigma)=(0,p,-q)=a_1\vec{s}_{j-1}+a_2\vec{s}_j+a_3\vec{s}_{j+1}$
in the basis ${\mathcal B}_j$. This basis is called
\textit{appropriate for the convergent}, when the vector
$\vec{a'}=\rho^{-{j\over 4}}\vec{a}$ is the shortest for some
positive integer $j$.
\end{defn}

\begin{exa}
From the table of convergents \cite{PLS} we take $p=1251,q=4813$ that is just
preceding the relatively big partial quotient $b_{11}=14$, so that $|\delta|<{1\over 14}$
in the estimate $p-q\sigma={\delta\over q}$.

Which $j$ take, to make the vector $\vec{a'}$ shortest? Here are
some results in the Table 1.

\begin{table}\label{tab1}
\begin{center}
\begin{tabular}{|r|r|c|}
\hline
$j$&$\vec{a}$&$|\vec{a'}|$ \\
\hline\hline
8&(20,-69,-33)&5.34 \\ \hline
9&(-9,27,20)&1.68 \\ \hline
10&(0,-7,-9)&0.39 \\ \hline
11&(-7,-9,0)&0.28 \\ \hline
12&(-30,-21,-7)&0.65 \\ \hline
13&(-111,-97,-30)&1.89 \\ \hline
\end{tabular}
\end{center}
\vskip5mm
\caption{Appropriate vector}
\end{table}

And we see that the appropriate $j$ and ${\mathcal B}_j$ to give
the shortest $\vec{a}'$ is $j=11$.
\end{exa}

\begin{thm} Let ${p\over q}$ be a convergent to $\sigma$,
${\mathcal B}_j$ its appropriate basis. Then for its reduced
coefficient vector we have
$$|\vec{a}'|<2.01.$$
\end{thm}

\begin{proof} We can write $p=q\sigma+{\delta\over q}$ with $|\delta|<1$
$$(0,p,-q)=p\vec{s}_0-q\vec{s}_1=p(\vec{h}+2\vec{g})-q(\sigma\vec{h}+2\sqrt{\rho}(\vec{g}\cos\theta-\vec{k}\sin\theta))$$
and when we rearrange the terms
$$(0,p,-q)={\delta\over q}\vec{h}+q\trd\sqrt{3}(\sqrt{3}\vec{g}+\vec{k})+{2\delta\over q}\vec{g},$$
we see what happens to either term under action of $S^{-j}$ with
growing $j$. The first term grows exponentially with $\rho^j$ in
the direction of the eigenvector $\vec{h}$, the second decreases
with $\rho^{-{j\over 2}}$ and rotates in the eigenplane, and the
last term also decreases

\begin{eqnarray}\label{a1}
\vec{a}=S^{-j}(0,p,-q)&=&\rho^j{\delta\over
q}\vec{h}+\rho^{-{j\over 2}}q
2\trd\sqrt{3}(\vec{g}\cos(j\theta+{\pi\over
6})+\vec{k}\sin(j\theta+{\pi\over 6}))\nonumber\\ &+&
\rho^{-{j\over 2}}{2\delta\over
q}(\vec{g}\cos(j\theta)+\vec{k}\sin(j\theta)).
\end{eqnarray}

But if we put $K_j=q\rho^{-{3j\over 4}}$
we get
\begin{eqnarray*}
\vec{a}=S^{-j}(0,p,-q)&=&\rho^{j\over 4}{\delta\over
K_j}\vec{h}+\rho^{j\over 4}K_j
2\trd\sqrt{3}(\vec{g}\cos(j\theta+{\pi\over
6})+\vec{k}\sin(j\theta+{\pi\over 6}))\\
 &+& \rho^{-{5j\over
2}}{2\delta\over K_j}(\vec{g}\cos(j\theta)+\vec{k}\sin(j\theta)).
\end{eqnarray*}
First we represent the vector ${\bf a}'=\rho^{-{j\over 4}}{\bf a}={\bf a}''+{\bf a}'''$ with
\begin{eqnarray*}
{\bf a}''&=&{\delta\over K_j}{\bf h}+2K_j\trd\sqrt{3}({\bf
g}\cos\alpha_j+{\bf k}\sin\alpha_j),\\
{\bf a}'''&=&\rho^{-{5j\over 4}}{2\delta\over K_j}({\bf g}\cos
j\theta+{\bf k}\sin j\theta),
\end{eqnarray*}
where $\alpha_j=j\theta+{\pi\over 6}$. As for ${\bf a}''$ we write the square of
its norm as a sum of three terms
\begin{eqnarray*}
|{\bf a}''|^2&=&{\delta^2|{\bf h}|^2\over K_j^2}+12\trs K_j^2
|{\bf g}\cos\alpha_j+{\bf k}\sin\alpha_j|^2+
4\delta\trd\sqrt{3}\langle {\bf h},{\bf g}\cos\alpha_j+{\bf k}\sin\alpha_j\rangle\nonumber \\
&=&T_1+T_2+T_3.
\end{eqnarray*}

The last term is estimated independently of $K_j$
\begin{equation}\label{T3}
|T_3|<4\cdot 1\cdot \trd\sqrt{3}\max_{\alpha\in\RR}\langle {\bf
h},{\bf g}\cos\alpha+{\bf k}\sin\alpha\rangle<0.894896.
\end{equation}
Denote
$$F(K_j,\alpha_j)=T_1+T_2={\delta^2|{\bf h}|^2\over K_j^2}+12\trs K_j^2 |{\bf g}\cos\alpha_j+{\bf k}\sin\alpha_j|^2$$
and to eliminate the dependence on $\alpha_j$ and $\delta$ we define another function
$$G(K_j)={|{\bf h}|^2\over K_j^2}+12\trs K_j^2 \max_{\alpha\in\RR}|{\bf g}\cos\alpha+{\bf k}\sin\alpha|^2$$
or inserting the numerical values
$$G(K_j)<{0.07873129\over K_j^2}+12.95559953 K_j^2=H(K_j),$$
observing that $F(K_j,\alpha_j)<H(K_j)$. The values of both
functions depend only on the choice of $j$, the variable $K_j$
assumes discrete values from a geometric series as
$K_j=q\rho^{-{3j\over 4}}$. The function $H(x)={a\over x^2}+b x^2$
as a function of continous variable $x>0$ features just one
minimum at $x_0=\root 4 \of {a\over b}$  with value
$H(x_0)=2\sqrt{ab}$, but the discrete variable $K_j$ shall almost
certainly miss this minimum point. We shall further denote by $x'$
the unique solution to the equation $H(x')=H(x'\rho^{-{3\over
4}})$. Indeed the equation $H(x')=H(x'\rho^{-{3\over 4}})$ reads
$${a\over x^2}+bx^2={a\over x^2\rho^{-{3\over 2}}}+bx^2\rho^{-{3\over 2}}$$
and we can easily solve it
$$x'=\root 4 \of {a\over b}\rho^{3\over 8}=x_0\rho^{3\over 8},$$
with the value
$$H(x')=\sqrt{ab}(\rho^{3\over 4}+\rho^{-{3\over 4}}).$$

However within the interval
$\lbrack x'\rho^{-{3\over 4}},x' \rbrack$ there is exactly one $K_j$ and this
defines also the choice of $j$.

To determine $j$ we have
$$ x'\rho^{-{3\over 4}}<K_j<x',$$
$$ x'\rho^{-{3\over 4}}<q\rho^{-{3j\over 4}}<x'.$$
Taking logarithms
$$\ln x'-{3\over 4}\ln\rho<\ln q-{3j\over 4}\ln\rho<\ln x'$$
and dividing by $(-{3\over 4}\ln\rho)$ we get
\begin{equation}\label{ocena}
1-{4\ln
x'\over 3\ln\rho}>j-{4\ln q\over 3\ln\rho}>-{4\ln x'\over
3\ln\rho},
\end{equation}
$$1+{4(\ln q - \ln x')\over 3\ln\rho}>j>{4(\ln q - \ln x')\over 3\ln\rho},$$
$$j=\left\lbrack 1+{4(\ln q - \ln x')\over 3\ln\rho}\right\rbrack \in \NN .$$

Thus we have
$$F(K_j,\alpha_j)<H(K_j)<H(x').$$

Inserting numerical values gives $x_0=0.279205$, $x'=0.462761$,
$x'\rho^{-{3\over 4}}=0.168457$,  $H(x')=3.142064$ and so
$F(K,\alpha)<3.142064$, which together with the estimate
(\ref{T3}) yields
$$|{\bf a}''|<\sqrt{3.142064+0.894896}<2.009219.$$

Let take $\varepsilon = 2.01 - 2.009219=0.000781$. We approximate
${\bf a}'''$
$$
|{\bf a}'''|<\rho^{-{5j\over 4}}{2\cdot1\over
K_j}\max_{\theta\in\RR}|{\bf g}\cos \theta+{\bf k}\sin \theta| <
\rho^{-{5j\over 4}}{1.649395\over x'\rho^{-{3\over
4}}}<\varepsilon
$$
from where we get $j\ge 6$ and the desired inequality follows for
these $j$.

If $j\le 5$ we get from the inequality (\ref{ocena}) condition on
$q$
$$
5-{4\ln q\over 3\ln\rho}\ge j-{4\ln q\over 3\ln\rho}>-{4\ln
x'\over 3\ln\rho}
$$

and $q$ has to be smaller than 73. There is only five convergents
with such $q$ and from the  Table 2 we see, that computed $|{\bf
a}'|$ satisfies our inequality.

\begin{table}[!htbp]\label{convxx}
\begin{center}
\begin{tabular}{|c|c|}
\hline
${p/q}$&$|\vec{a'}|$ \\
\hline\hline ${1/ 3}$&1.151 \\ \hline ${1/ 4}$&0.581 \\
\hline ${6/ 23}$&0.928 \\ \hline ${7/ 27}$&0.870 \\ \hline ${13/ 50}$&0.415 \\
\hline
\end{tabular}
\end{center}
\vskip5mm
\caption{First five convergents}
\end{table}

Thus we can confirm
$$|{\bf a}'|<2.01$$
\end{proof}

\begin{rem} Over the first 10000 convergents we numerically find that $|\vec{a}'|<1.753.$
The adjacent Figure \ref{fig:gr1} shows the statistics in dots,
$|\vec{a}'|$, for these convergents.
\end{rem}

\begin{figure}[!htbp]
\centering
\includegraphics[width=120mm]{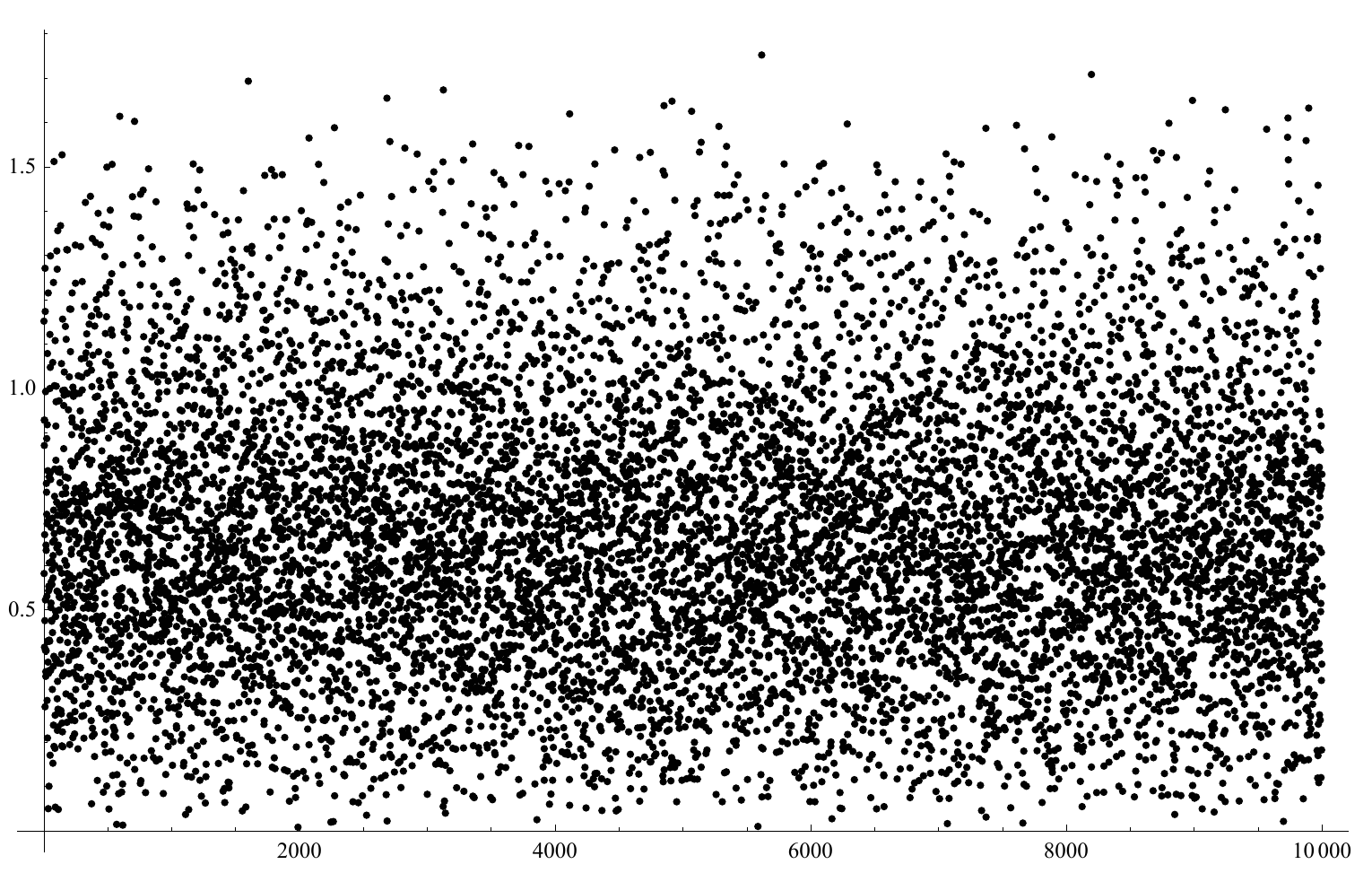}
\caption{Length of vectors $\vec{a}'$}\label{fig:gr1}
\end{figure}

Also the other way round, if coefficients are small enough, we are
dealing with a convergent. The following theorem is however a rather
coarse one.

\begin{thm} Let ${\bf a}=(a_1,a_2,a_3)\in\ZZ^3$ be the coefficient vector for the basis
${\mathcal B}_j$, $j\in \NN$ such that $\langle {\bf a},{\bf
s}^*_j\rangle=0$ and $|{\bf a}|<{1\over 3}\rho^{j\over 4}$. Then
the resulting vector $a_1{\bf s}_{j-1}+a_2{\bf s}_j+a_3{\bf
s}_{j+1}=(0,p,-q)$ yields a continous fraction convergent ${p\over
q}$.
\end{thm}

\begin{proof}

We multiply the equation (\ref{a1}) with the conjugate vectors
${\bf h}^*, {\bf g}^*, {\bf k}^*$ to obtain

\begin{eqnarray}\label{hz}\langle{\bf h}^*, {\bf a}\rangle&=&\rho^j{\delta\over q}\langle{\bf h}^*, {\bf
h}\rangle\\
\label{gz}\langle{\bf g}^*, {\bf a}\rangle&=&\rho^{-{j\over 2}}q
2\trd\sqrt{3}\langle{\bf g}^*, {\bf g}\rangle
\cos(j\theta+{\pi\over 6})+\rho^{-{j\over 2}}{2\delta\over
q}\langle{\bf g}^*, {\bf g}\rangle\cos(j\theta)\\
\label{kz}\langle{\bf k}^*, {\bf a}\rangle&=&\rho^{-{j\over 2}}q
2\trd\sqrt{3}\langle{\bf k}^*, {\bf k}\rangle
\sin(j\theta+{\pi\over 6})+\rho^{-{j\over 2}}{2\delta\over
q}\langle{\bf k}^*, {\bf k}\rangle\sin(j\theta)
\end{eqnarray}
Inserting the conditions of the theorem, we can estimate the
quantitie ${|\delta|\over q}$ using (\ref{hz}):
\begin{equation}\label{deltaq}{|\delta|\over q}<\rho^{-{3j\over 4}}{|{\bf h}^*|\over
3M}.
\end{equation}
Using (\ref{gz}) and (\ref{kz}) we get
$$q\cdot 2\trd\sqrt{3}\cos(j\theta+{\pi\over 6})M=\rho^{j\over 2}
\langle{\bf g}^*,{\bf a}\rangle- 2{\delta\over q}M\cos(j\theta),$$
$$q\cdot 2\trd\sqrt{3}\sin(j\theta+{\pi\over 6})M=-\rho^{j\over 2}
\langle{\bf k}^*,{\bf a}\rangle- 2{\delta\over q}M\sin(j\theta).$$
We square and add up the last two equations to eliminate the sines
and cosines

$$q^2(2\trd\sqrt{3}M)^2=\rho^j(\langle{\bf g}^*,{\bf a}\rangle^2+\langle{\bf k}^*,{\bf a}\rangle^2) +
\rho^{j\over 2}{4\delta\over q}M(\langle{\bf k}^*,{\bf
a}\rangle\cos(j\theta)- \langle{\bf g}^*,{\bf
a}\rangle\sin(j\theta))+{4\delta^2\over q^2}M^2$$

and estimate

$$q^2(2\trd\sqrt{3}M)^2<\rho^{3j\over 2}{|{\bf g}^*|^2+|{\bf k}^*|^2\over 9}+
{4|{\bf h}^*|\over 9}(|{\bf g}^*|+|{\bf k}^*|)+\rho^{-{3j\over 2}}
{4|{\bf h}^*|^2\over 9}.$$

Using (\ref{deltaq}) and last inequality we can estimate
$$
|\delta|<{|{\bf h}^*|\over 18 M^2\trd\sqrt{3}}\sqrt{|{\bf
g}^*|^2+|{\bf k}^*|^2+4\rho^{-{3j\over 2}}|{\bf h}^*|(|{\bf
g}^*|+|{\bf k}^*|)+4\rho^{-3j}|{\bf h}^*|^2}
$$

which is   smaller than $0.48$ for $j\ge 2$. For $j=1$ condition
$|{\bf a}|<{1\over 3}\rho^{1\over 4}<0.47$ implies nonexistance of
such integer vector. So we show $|\delta|<{1\over2}$, that is
enough for our conclusion.
\end{proof}

From the proof we can also infer the implications of an even smaller
$|{\bf a}|$ on the $|\delta|$ and consecutively the next partial quotient $B$.

\begin{lem}
If in the above theorem $|{\bf a}|<{\Delta\over 3}\rho^{j\over 4},
\Delta<1$, we can estimate the next partial quotient $B$ to the
convergent ${p\over q}$ by
$$B>{2\over\Delta^2}-2.$$
\end{lem}

\begin{proof}
From the well known estimate \cite{P}
$$
{1\over(B+2)q^2}<\left|{p\over q}-\sigma\right|<{1\over B q^2}
$$
we find ${1\over B+2}<|\delta|$ or
$$
B>{1\over |\delta|}-2.
$$
But from the above proof, if instead of ${1\over3}$, we put
${\Delta\over3}$, we have
$|\delta|<\Delta^2\cdot0.48<\Delta^2\cdot{1\over2}$ and the
estimate from lemma follows.
\end{proof}

\begin{exa}In our numerical experiment we found for $j=750$,
$\Delta=0.03906$, next $B=4941=b_{619}$ (well known big partial
quotient \cite{PLS}) with lemma suggesting $B>1308$.
\end{exa}

\subsection{The problem of the shortest lattice vector}
 Our case the lattice
$\Lambda_j=\{\vec{a}\in\ZZ^3,\vec{a}\perp\vec{s}^*_j\}$. Gauss reduction process
mimics the euclidian algorithm. Let's have a basis $\vec{z}_1,\vec{z}_2$, such that
$|\vec{z}_1|<|\vec{z}_2|$. Choose $k$ so that
$$-{1\over 2}|\vec{z}_1|^2<\langle \vec{z}_2-k\vec{z}_1,\vec{z}_1\rangle\le{1\over 2}|\vec{z}_1|^2$$
so $k\in\ZZ$ is the nearest integer to
${\langle \vec{z}_2,\vec{z}_1\rangle\over|\vec{z}_1|^2}$.

Now, set the new $\vec{z}_2:=\vec{z}_2-k\vec{z}_1$ and compare: if
$|\vec{z}_1|<|\vec{z}_2|$ the process terminates, our shortest vector is
$\vec{z}_1$, else we interchange $\vec{z}_1\leftrightarrow\vec{z}_2$ and start again.
In some steps we get the shortest lattice vector \cite{St}.

\begin{rem}
In higher dimensions the so called LLL-algorithm \cite{St},
\cite{LLL} is similar to Gramm-Schmidt orthogonalization to
generalize the Gauss process.
\end{rem}

\begin{exa}
Here is how we carried out this process for $j=7$, i.e.
$\vec{z}_1=\vec{s}_{-6}$, $\vec{z}_2=\vec{s}_{-7}$ and we have the
shortest vector $\vec{a}=(-7,1,0)$. Combining
$-7\vec{s}_6+1\vec{s}_7=$ $(0,59,-227)$ we read off the quotient
${59\over 227}$ which does appear in the sequence of approximants
(Table 3).

\begin{table}[!htbp]\label{GS}
\begin{center}
\begin{tabular}{|r|r|r|r|r|r|}
\hline
$n$&$\vec{z}_1$&$\vec{z}_2$&$|\vec{z}_1|^2$&$\langle\vec{z}_1,\vec{z}_2\rangle$&$k$ \\
\hline\hline 1&(681,577,177)&(2620,2220,681)&828019&3185697&4 \\
\hline 2&(-104,-88,-27)&(681,577,177)&19289&-126379&-7 \\ \hline
3&(-47,-39,-12)&(-104,-88,-27)&3874&8644&2 \\ \hline
4&(-10,-10,-3)&(-47,-39,-12)&209&896&4 \\ \hline
5&(-7,1,0)&(-10,-10,-3)&50&60&1 \\ \hline
6&(-3,-11,-3)&(-7,1,0)&139&10&0 \\ \hline
\end{tabular}
\vskip5mm
\caption{Shortest lattice vector}
\end{center}
\end{table}
\end{exa}

We carried out the shortest vector algorithm for $j=2$ until
$j=1000$. The resulting $|\vec{a}'|$ were, as shown in the dotted
Figure \ref{fig:gr2}, all below 1. Only 21 of them did not result
in continued fraction approximants (marked with squares).
\begin{figure}[!htbp]
\centering
\includegraphics[width=100mm]{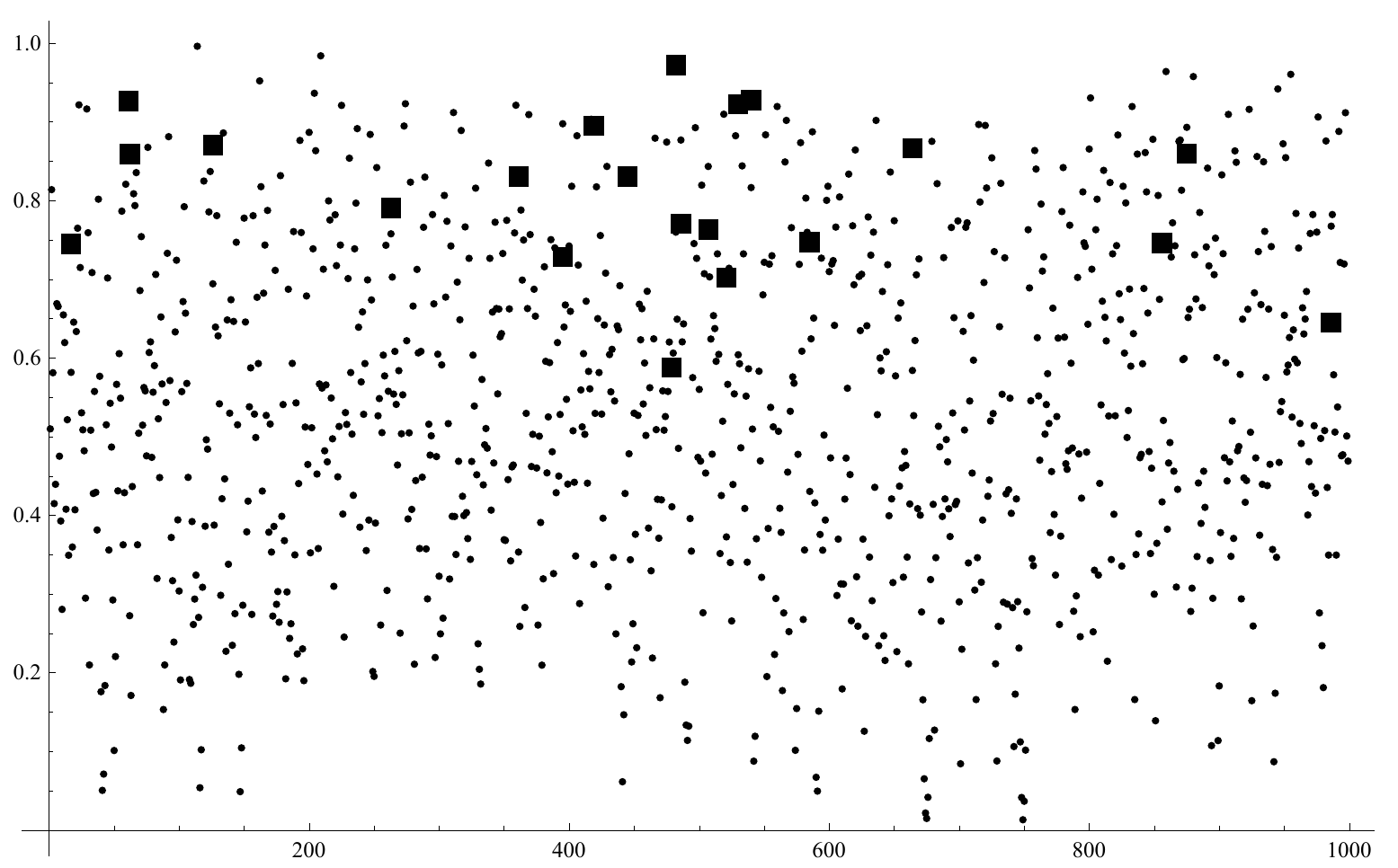}
\caption{Length of vectors $\vec{a}'$ resulting from the shortest vector algorithm}\label{fig:gr2}
\end{figure}

\section{$\chi^2$-test of distribution of partial quotients}

We applied $\chi^2$-test to compare observed frequencies of partial quotients of $\trd$ with theoretical
frequencies $P(b_n=k)=\log_2{(k+1)^2\over k(k+2)}$. Using \cite{math} we computed 75\;000 partial quotients and
divided them into $R$ groups consisting of numbers 1, 2, 3, \ldots, $R-1$ and of all numbers over $R-1$. Let $O_i$ be the observed frequency
of the $i^{th}$ group and $E_i$ its expected frequency. The value of the test statistic is
$$X^2=\sum_{i=1}^R \frac{(O_i-E_i)^2}{E_i}.$$  If the partial quotients the hypothesized distribution, $X^2$ has, approximately, a $\chi^2$ distribution with $R-1$ degrees of freedom. The resulting P-values for different degrees of freedom are shown in the Figure \ref{hi-test}. Since all the P-values are above $0.05$ we can not reject the hypothesis that the partial quotients of $\trd$ follow the distribution law of Kuzmin.

\begin{figure}[!htbp]
\centering
\includegraphics[width=100mm]{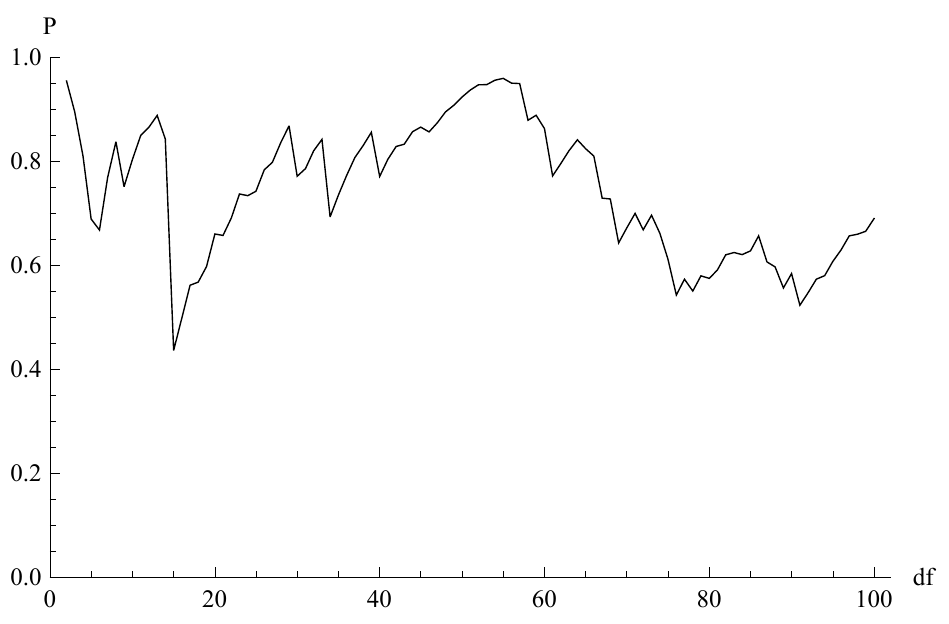}
\caption{P-value for different degrees of freedom}\label{hi-test}
\end{figure}


\end{document}